\documentclass[11pt]{amsart}
\usepackage[top=1.25in, bottom=1in, left=1in, right=1in]{geometry}
\usepackage{graphicx}
\usepackage[colorlinks=true, citecolor=blue]{hyperref}
\usepackage{calc}

\usepackage{amsmath, amsthm}
\usepackage{amsfonts}
\usepackage{bbm}
\usepackage{mathrsfs}
\usepackage{amssymb}
\usepackage{mathtools}
\usepackage[all]{xy}
\usepackage{tikz}
\usepackage{tikz-cd}
\usetikzlibrary{matrix,arrows,decorations.pathmorphing}

\newtheorem{theorem}{Theorem}[section]
\newtheorem{lemma}[theorem]{Lemma}
\newtheorem{prop}[theorem]{Proposition}

\theoremstyle{definition}

\newtheorem{example}[theorem]{Example}

\theoremstyle{remark}

\numberwithin{equation}{section}

\newcommand{\R}{\mathbb{R}}
\newcommand{\M}{\mathcal{M}}

\newcommand{\Q}{\mathbb{Q}}

\newcommand{\wt}{\widetilde}
\newcommand{\ra}{\rightarrow}
\renewcommand{\P}{\mathbb{P}}

\newcommand{\calX}{\mathcal{X}}

\newcommand{\calH}{\mathcal{H}}

\newcommand{\ov}{\overline}
\renewcommand{\o}{\mathcal{O}}

\newcommand{\calP}{\mathcal{P}}
\newcommand{\calC}{\mathcal{C}}

\begin{document}

\title[Effective divisors in the projectivized Hodge bundle]{Effective divisors in the projectivized Hodge bundle}

\author[Iulia Gheorghita]{Iulia Gheorghita}
\address{Department of Mathematics, Boston College, Chestnut Hill, MA 02467}
\email{gheorgiu@bc.edu}

\date{\today}

\begin{abstract} 
We compute the class of the closure of the locus of canonical divisors in the projectivization of the Hodge bundle $\P\ov{\calH}_g$ over $\ov{\M}_g$ which have a zero at a Weierstrass point. We also show that the strata of canonical and bicanonical divisors with a double zero span extremal rays of the respective pseudoeffective cones.

\end{abstract}

\maketitle

\section{Introduction}
The Hodge bundle $\calH_g$ over $\M_g$ parametrizes pairs $(C, \omega)$ where $C$ is a smooth genus $g$ curve and $\omega$ is a holomorphic abelian differential on $C$. If $\mu=(m_1, \dots, m_n)$ is a partition of $2g-2$, we denote by $\calH_g(\mu)$ the stratum consisting of $(C, \omega)$ where $\mu$ describes the multiplicities of the zeros of $\omega$. This describes a natural stratification on the complement of the zero section of $\calH_g$. We can projectivize $\calH_g$ to get a $(4g-4)$-dimensional space $\P\calH_g$ which parametrizes canonical divisors on smooth genus $g$ curves. The Hodge bundle also extends over the boundary of $\ov{\M}_g$, where the fiber over a nodal curve consists of stable differentials - that is, differentials that have at worst simple poles at the nodes with opposite residues on the two branches of a node. We denote the projectivization of this bundle by $\P\ov{\calH}_g$. We will denote by $\P\ov{\calH}_g(\mu)$ the closure of the strata in $\P\ov{\calH}_g$. 

A pair $(C, \omega)$ can also be realized as a plane polygon with sides identified by translation. The action of $\text{GL}_2^+(\R)$ on the plane induces an action on the strata, which is called Teichm\"uller dynamics. This provides one source of motivation for studying these objects (see \cite{zorich2006flat} and \cite{chen2017teichmuller}). 

On the other hand, effective divisors defined by geometric conditions have been widely studied since Harris and Mumford used them to determine the Kodaira dimension of $\M_g$ \cite{harris1982kodaira}. The class $W$ of the closure of the locus in $\ov{\M}_{g, 1}$ of curves with a marked Weierstrass point was first calculated in \cite{cukierman1989}. The class of the divisorial stratum $\P\ov{\calH}_g(2, 1^{2g-4})$ in $\text{Pic}(\P\ov{\calH}_g) \otimes \Q$ was computed in \cite{korotkin2011tau}. Recently, Mullane computed the classes of many effective divisors arising from the strata of differentials in order to study the effective cone of $\ov{\M}_{g,n}$ \cite{mullane2017effective}. In this paper we consider the analog of the Weierstrass divisor in the projectivization of the Hodge bundle: \[ D := \ov{\Big\{(C, \omega) \in \P\calH_g \; | \;  \text{div} \; \omega \; \text{contains a Weierstrass point} \Big\} }\subset \P\ov{\calH}_g\] and compute its class in $\text{Pic}(\P\ov{\calH}_g) \otimes \Q$.

\begin{theorem} In $\text{Pic}(\P\ov{\calH}_g) \otimes \Q$, \[ [D] = -(g-1)g(g+1)\eta + 2(3g^2+2g+1)\lambda - \frac{g(g+1)}{2}\delta_0 + \sum_{i=1}^{\lfloor g/2 \rfloor} (g+3)i(i-g)\delta_i\] where $\eta := \o_{\P\ov{\calH}_g}(-1)$ and $\lambda$, $\delta_0$, and $\delta_i$ denote the respective pullbacks from $\ov{\M}_g$.

\end{theorem}

In order to prove Theorem 1.1, we make use of the incidence variety compactification of the strata $\P\calH_g(\mu)$ provided in \cite{BCGGM}. We denote the incidence variety compactification of the strata by $\ov{\calP}(\mu)$. In particular, we will often pull back test curves to $\ov{\calP}(1^{2g-2})$, which allows us to separate the zeros of a differential and also make use of the class of the standard Weierstrass divisor $W \subset \ov{\M}_{g,1}$.

In this paper we also investigate the extremality of divisors in the pseudoeffective cone. We denote by $\ov{\mathcal{Q}}_g$ over $\ov{\M}_g$ the bundle of quadratic differentials, i.e., $\pi_*(\omega_{\pi}^{\otimes 2})$ where $\pi: \ov{\calC}_g \ra \ov{\M}_g$ is the universal curve and $\omega_\pi$ the relative dualizing sheaf. If $\mu = (d_1, \dots, d_n)$ is now a partition of $4g-4$, we denote by $\ov{\mathcal{Q}}_g(\mu)$ the stratum parametrizing quadratic differentials where $\mu$ describes the multiplicities of the zeros. Using results from \cite{nonvaryinglyapsum} and \cite{quadlyapsum}, in the second part of the paper we prove the following theorem. 

\begin{theorem} The divisors $\P\ov{\calH}_g(2, 1^{2g-4}) \subset \P\ov{\calH}_g$ and $\P\ov{\mathcal{Q}}_g(2, 1^{4g-6}) \subset \P\ov{\mathcal{Q}}_g$ span extremal rays of the respective pseudoeffective cones.
\end{theorem}

We prove this by observing that Teichm\"uller curves, which are dense in the strata, have negative intersection with the classes of these divisors. Much work has been done to determine the extremality of $W \subset \ov{\M}_{g,1}$ for small $g$ (\cite{rulla2001birational}, \cite{chen2013strata}, \cite{jensen2013birational}, \cite{jensen2012rational}), yet the question remains open for $g \geq 6$. We hope that our computation of the class of the related divisor $D$ will contribute to this discussion.

\subsection*{Acknowledgements} I would like to thank my advisor Dawei Chen for introducing me to these ideas and for many helpful discussions. During the preparation of this work I was partially supported by NSF CAREER grant DMS-1350396.

\section{Preliminaries}
Recall that \[ \text{Pic}(\P\ov{\calH}_g) \otimes \Q = \langle \eta, \lambda, \delta_0, \dots, \delta_{\lfloor g/2 \rfloor}\rangle\] where $\eta := \o_{\P\ov{\calH}_g}(-1)$ and the remaining classes are the pullbacks from $\ov{\M}_g$. Let $\pi: \ov{\calC}_g \ra \ov{\M}_g$ be the universal curve and $\omega_\pi$ the relative dualizing sheaf. We may also replace $\lambda$ in the above expression with the pullback of $\kappa$ from $\ov{\M}_g$, where $\kappa = \pi_*c_1^2(\omega_\pi)$. This is because $\kappa = 12\lambda - \delta$, where $\delta$ is the total boundary class. 

Let $S \subseteq \{1, \dots, n\}$. Recall also that \[ \text{Pic}(\ov{\M}_{g,n}) \otimes \Q = \langle \lambda, \psi_1, \dots, \psi_n, \delta_0, \delta_{i; S}\rangle\] where $0 \leq i \leq \lfloor g/2 \rfloor$ and $\delta_{i; S}$ denotes the class of the divisor whose general points parameterize one-nodal curves whose genus $i$ component contains the markings labelled by the subset $S$. Note that if $i = 0$, then $|S| \geq 2$. For more details, see \cite{arbarello1987picard}.

Computing the degree of $\eta$ on test curves is made easier by a relation which we will now describe. The reader can see \cite{chen2018positivity} for more details. Let $\pi: \calX \ra B$ be a one-parameter family of pointed stable differentials in $\ov{\calP}(\mu)$ whose generic fiber is smooth. If $\calX$ is singular we replace it by its minimal resolution. Let $S_1, \dots, S_n$ be the distinct sections which mark the zeros and poles of the differentials parameterized by this family and let $\omega$ be the relative dualizing line bundle class of $\pi$. Moreover, let $V$ be the union of the irreducible components where the parameterized differentials are identically zero. Then, since $\pi^*\eta$ has zeros or poles along the $S_i$ with multiplicity $m_i$ and zeros along $V$, we have a relation of divisor classes in $\calX$ \[ \pi^*\eta = \omega - \sum_{i=1}^n m_iS_i - V.\] 

We will also regularly make use of the incidence variety compactification of the strata. For a partition $\mu=(m_1, \dots, m_n)$ of $2g-2$, define \[ \calP(\mu) := \Big\{ (X, \omega, z_1, \dots, z_n) \in \P\calH_{g,n} \; | \; \text{div} \;\omega = \sum_{i=1}^n m_iz_i \Big\}\] where $\P\calH_{g,n}$ denotes the projectivized Hodge bundle over $\M_{g,n}$ parametrizing pointed stable differentials (with ordered marked points). The incidence variety compactification $\ov{\calP}(\mu)$ is defined to be the closure of $\calP(\mu)$ inside $\P\ov{\calH}_{g,n}$. We refer the reader to \cite{BCGGM} for more details. Note that we will use the notation $\ov{\calP}(\mu)$ for the incidence variety compactification of a stratum, whereas in \cite{BCGGM} it is denoted $\P\Omega\ov{\M}^{\text{inc}}_{g, n}(\mu)$.

\section{The class of $D$}

In this section we will prove Theorem 1.1. Recall that \[ D := \ov{\Big\{(C, \omega) \in \P\calH_g \; | \;  \text{div} \; \omega \; \text{contains a Weierstrass point} \Big\} }\subset \P\ov{\calH}_g.\] We will show that in $\text{Pic}(\P\ov{\calH}_g) \otimes \Q$, \[ [D] = -(g-1)g(g+1)\eta + 2(3g^2+2g+1)\lambda - \frac{g(g+1)}{2}\delta_0 + \sum_{i=1}^{\lfloor g/2 \rfloor} (g+3)i(i-g)\delta_i.\]

From now on we will use the notation $D$ to refer to the class of $D$ in $\text{Pic}(\P\ov{\calH}_g) \otimes \Q$. Before beginning the proof we introduce some notation and prove a lemma. Note that we have the following morphisms between the various moduli spaces 

\[\begin{tikzcd} \ov{\calP}(1^{2g-2}) \arrow{r}{\varphi} \arrow{d}{\pi}
						& \ov{\M}_{g, 2g-2} \arrow[d, "f_i"] \\
			\P\ov{\calH}_g
						& \ov{\M}_{g, 1}
\end{tikzcd}\]

\noindent where $\pi$ forgets the marked points, $\varphi$ forgets the differential, and $f_i$ forgets all but the $i$th marked point. We denote by $W \subset \ov{\M}_{g,1}$ the standard Weierstrass divisor. We also consider the divisiors \begin{align*} & D_{2g-2} := \ov{\Big\{(X, \omega, z_1, \dots, z_{2g-2}) \in \calP(1^{2g-2}) \;  | \;  \text{some} \; z_i \; \text{is a Weierstrass point} \Big\}} \subset \ov{\calP}(1^{2g-2}) \\
& W_{n} := \ov{\Big\{ (C, z_1, \dots, z_{n}) \in \M_{g, n} \;  |  \; \text{some} \; z_i \; \text{is a Weierstrass point}\Big\}} \subset \ov{\M}_{g,n}. \end{align*}

Note that $D_{2g-2}$ is the proper transform of $D$ under the morphism $\pi$ and that $\pi_*D_{2g-2} = (2g-2)! D$. Similarly, $\varphi^*W_{2g-2} = D_{2g-2} + E$ for some collection $E$ of irreducible boundary divisors of $\ov{\calP}(1^{2g-2})$ with image in $W_{2g-2}$. Over the locus of smooth curves the pullback of $W_{2g-2}$ is simply $D_{2g-2}$, so $E$ must contain just pointed stable differentials with nodal underlying curves. We will see that $E$ is indeed nonempty.

Let $m \geq 0$. We will denote by $E_i(m)$ the closure of the locus in $\ov{\calP}(1^{2g-2})$ of one-nodal curves with a component $C_1$ of genus $i$, attached to a component $C_2$ of genus $g-i$ at a point $q_1 \sim q_2$, having twisted differentials $\eta_1$ and $\eta_2$ of types $(1^{2i-m-2}, m)$ and $(1^{2g-2i+m}, -(m+2))$ respectively, where $\eta_1$ has its order $m$ zero at $q_1$ and $\eta_2$ has its pole at $q_2$ (see Figure 1). This locus is indeed codimension one in $\ov{\calP}(1^{2g-2})$. Note that when $i \geq 2$ imposing the condition that $m \geq i$ will ensure that $q_1$ is a Weierstrass point of $C_1$ and so by \cite[Theorem 5.45]{harris1998moduli} that every point of $C_2$ is a limit Weierstrass point when smooth genus $g$ curves degenerate to $C_1 \cup_q C_2$. Thus, in this case $E_i(m) \subset E$.  

\begin{figure}[h!]
	\centering
	\def \svgscale{.28}
	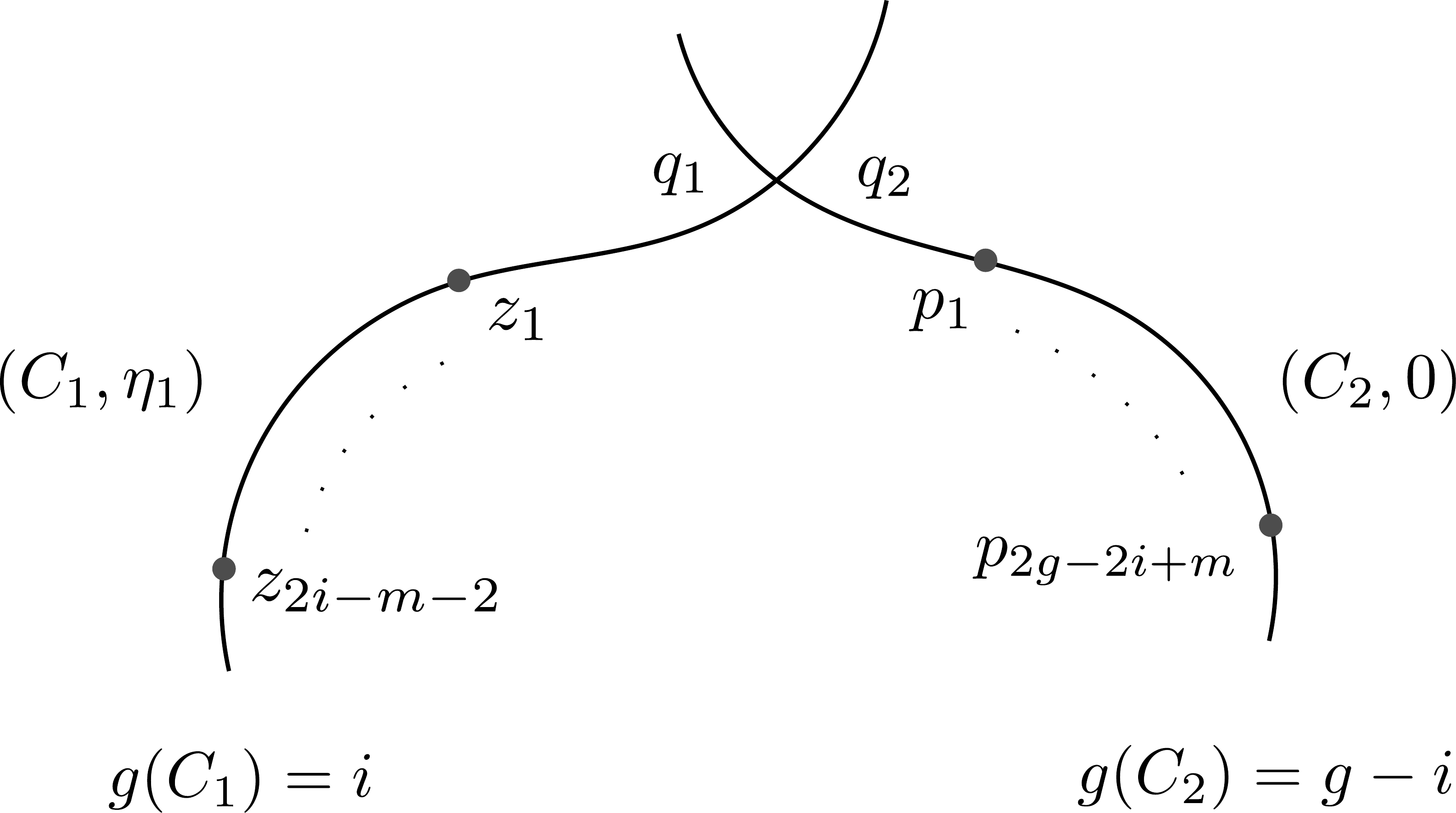
	\caption{Pointed stable differentials in $E_i(m)$ have the form illustrated above.}
\end{figure}

\begin{lemma} Let $E_i(m)$ be the locus in $\ov{\calP}(1^{2g-2})$ described above. When $m \leq i-2$, $i \geq 2$ and when $m=0$, $i=1$, the irreducible divisor $E_i(m)$ is not in $E$.
\end{lemma} 

\begin{proof} Assume $m \leq i-2$. It suffices to find a pointed stable differential in $E_i(m)$ whose zeros avoid the limit Weierstrass points. First, note that a general $(C_2, q_2, p_1, \dots, p_{2g-2i+m}) \in \varphi(\calP(-(m+2), 1^{2g-2i+m}))$ has $q_2$ not a Weierstrass point of $C_2$; indeed, we may choose any point on $C_2$ to be the pole. This condition ensures that not every point on $C_1$ is a limit Weierstrass point \cite[Theorem 5.45]{harris1998moduli}. Similarly, by \cite[Theorem 6.7]{chen2017degenerations} a general $(C_1, q_1, z_1, \dots, z_{2i-m-2}) \in \varphi(\calP(m, 1^{2i-m-2}))$ has $q_1$ not a Weierstrass point of $C_1$. Again, this condition ensures that not every point of $C_2$ is a limit Weierstrass point. 

By the description of limit Weierstrass points in \cite[Theorem 5.45]{harris1998moduli}, we know that these points form a finite set on both $C_1$ and $C_2$. Denote by $w$ a limit Weierstrass point in $C_2$. Then, we may choose a general differential on $C_2$ away from the finite set of hyperplanes of the form \[\P H^0(K_{C_2}((m+2)q_2 - w))  \subset \P H^0(K_{C_2}((m+2)q_2)) \cong \P^{g-i+m}.\] 

By Riemann-Roch and the condition that $m \leq i -2$ we have that $h^0(K_{C_1}(-mq_1)) \geq 2$. We will show for a general $C_1$ and $q_1$, that this linear series is base point free. If $r$ were a base point, then $h^0(mq_1+r) = h^0(mq_1) + 1$. This implies that $h^0(mq_1 + r) \geq 2$. Thus we can form an $m+1$ degree cover $C_1 \ra \P^1$ with a fiber of type $mq_1 + r$. The dimension of the Hurwitz scheme parametrizing genus $i$ curves along with degree $m+1$ covers of $\P^1$ with this specified ramification type is bounded by $3g-3$. As this is smaller than $\dim \M_{g, 1}$, we have that for general $(C_1, q_1)$, $h^0(mq_1 + r) = 1$ and hence that $|K_{C_1}(-mq_1)|$ has no base points. This means that for such an appropriate general choice of $(C_1, q_1)$, $\P H^0(K_{C_1}(-mq_1 - r))$ is indeed a hyperplane of $\P H^0(K_{C_1}(-mq_1))$. We may now choose a general differential on $C_1$ such that it avoids a finite set of hyperplanes of the form $\P H^0(K_{C_1}(-mq_1 - w))$ for $w$ a limit Weierstrass point of $C_1$. 

Now assume that $m=0$ and $i=1$. In this case, we simply need to choose some $(C_2, q_2, z_1, \dots, z_{2g-2}) \in \varphi(\calP(-2, 1^{2g-2}))$ such that the $z_i$ avoid the limit Weierstrass points of the curve. We do this precisely as above. 

\end{proof}

\begin{proof}[Proof of Theorem 1.1]
First, note that in genus 2, the divisor $D$ is simply $\P\ov{\calH}_2(2)$. When $g=2$ we can use the relation $\lambda = \frac{1}{10}\delta_0 + \frac{1}{5}\delta_1$ to see that the class of $\P\ov{\calH}(2)$ computed in \cite{korotkin2011tau} agrees with the class of $D$ given in Theorem 1.1. From now on assume $g \geq 3$. Let $\eta := \o_{\P\ov{\calH}_g}(-1)$ and write \[ D = a\eta + b\lambda + \sum_{i=0}^{\lfloor g/2 \rfloor} c_i\delta_i.\] \\

\noindent \textit{Test curve $A$.} \\

\indent Let $C$ be a general genus $g$ curve canonically embedded in $\P^{g-1}$. Let $\Lambda \cong \P^{g-3}$ be a fixed general subspace and consider the one-dimensional family of hyperplanes containing $\Lambda$. These hyperplanes are parametrized by $A \cong \P^1$ and cut out canonical divisors on $C$. As there are $(g-1)g(g+1)$ Weierstrass points on a general genus $g$ curve, we have $A \cdot D = (g-1)g(g+1)$. Moreover, since $A \subset \P{\ov{\calH}}_g$ is simply a line in the fiber above the point $C \in \M_g$, $A \cdot \eta = -1$. Clearly, $A \cdot \delta_i = 0$ for all $i \geq 0$ and by the projection formula $A \cdot \lambda = 0$.  This implies that \[a = -(g-1)g(g+1).\]

\indent We will now find the coefficient $b$. The strategy is to write $D$ as follows: \[ D = a'\eta + b'[\P\ov{\calH}_g(2, 1^{2g-4})] + \sum_{i=0}^{\lfloor g/2 \rfloor} c'_i\delta_i\] and to then find the class of $D$ in the stratum $\P\calH_g(1^{2g-2})$ above just the smooth locus $\M_g$. This will give us the coefficient $a'$ above. Then using the class \[ \P\ov{\calH}_g(2, 1^{2g-4}) = 24\lambda - (6g-6)\eta - 2\delta_0 - 3\sum_{i=1}^{\lfloor g/2 \rfloor} \delta_i\] \cite{korotkin2011tau} we extract the coefficient of $b'$ using our knowledge of $a$. \\

\indent Consider again the following morphisms: 
\[\begin{tikzcd} \calP(1^{2g-2}) \arrow[hookrightarrow]{r}{\varphi} \arrow{d}{\pi}
						& \M_{g, 2g-2} \arrow[d, "f_i"] \\
			\P\calH_g
						& \M_{g, 1}
\end{tikzcd}\]

\noindent Over the locus of smooth curves, we have that \[ W_{2g-2} = \sum_{i=1}^{2g-2}f_i^*W = \sum_{i=1}^{2g-2}\Big(\frac{g(g+1)}{2}\psi_i - \lambda\Big)\] where in the second equality we use the class of $W$ found in \cite{cukierman1989}. Note that $\varphi^*\psi_i = \frac{\eta}{2}$. We can see this by intersecting the relation \[ \omega = f^*\eta + \sum_{i=1}^{2g-2} S_i\] with $S_i$ and applying $f_*$, where $f: \calX \ra B$ is a family of pointed stable differentials (with underlying smooth curves and disjoint sections $S_i$). Since $\omega \cdot S_i = -S_i^2$, $\deg_{f}\psi_i = f_*(\omega \cdot S_i)$, and $\deg_{f}\eta = f_*(f^*\eta \cdot S_i)$ this gives the relation. So, \[ \varphi^*W_{2g-2} = \frac{(g-1)g(g+1)}{2}\eta - (2g-2)\lambda.\] Since $D|_{\P\calH(1^{2g-2})} = \pi_*\varphi^*W_{2g-2}/(2g-2)!$ and $\lambda = \frac{g-1}{4}\eta$ on $\P\calH_g(1^{2g-2})$ (see \cite[Section 3.4]{eskin2014sum}), \[ D|_{\P\calH_g(1^{2g-2})} = \frac{(g^2+1)(g-1)}{2}\eta.\] Thus, \[ a' =  \frac{(g^2+1)(g-1)}{2}. \] Using that $a = -(g-1)g(g+1)$, we find that \[ b' = \frac{3g^2+2g+1}{12}\] and therefore that \[ b = 2(3g^2+2g+1).\]  \\ 

\noindent \textit{Test curve $B$.} \\

\indent Consider the Hodge bundle $\ov{\calH}_1$ over a general pencil of plane cubics parametrized by $B \cong \P^1$. Since $\lambda$ has degree 1 on this family of curves, a section $s: B \ra \ov{\calH}_1$ will assign to exactly one cubic in the family the zero differential. Now attach this pencil of plane cubics with the corresponding differential given by $s$ to a general genus $g-1$ curve $C_2$ with a fixed general nonzero differential $\omega_2$ (see Figure 2). We call the node $q_1$ on the genus 1 curves and $q_2$ on $C_2$. 

\begin{figure}[h!]
	\centering
	\def \svgscale{.27}
	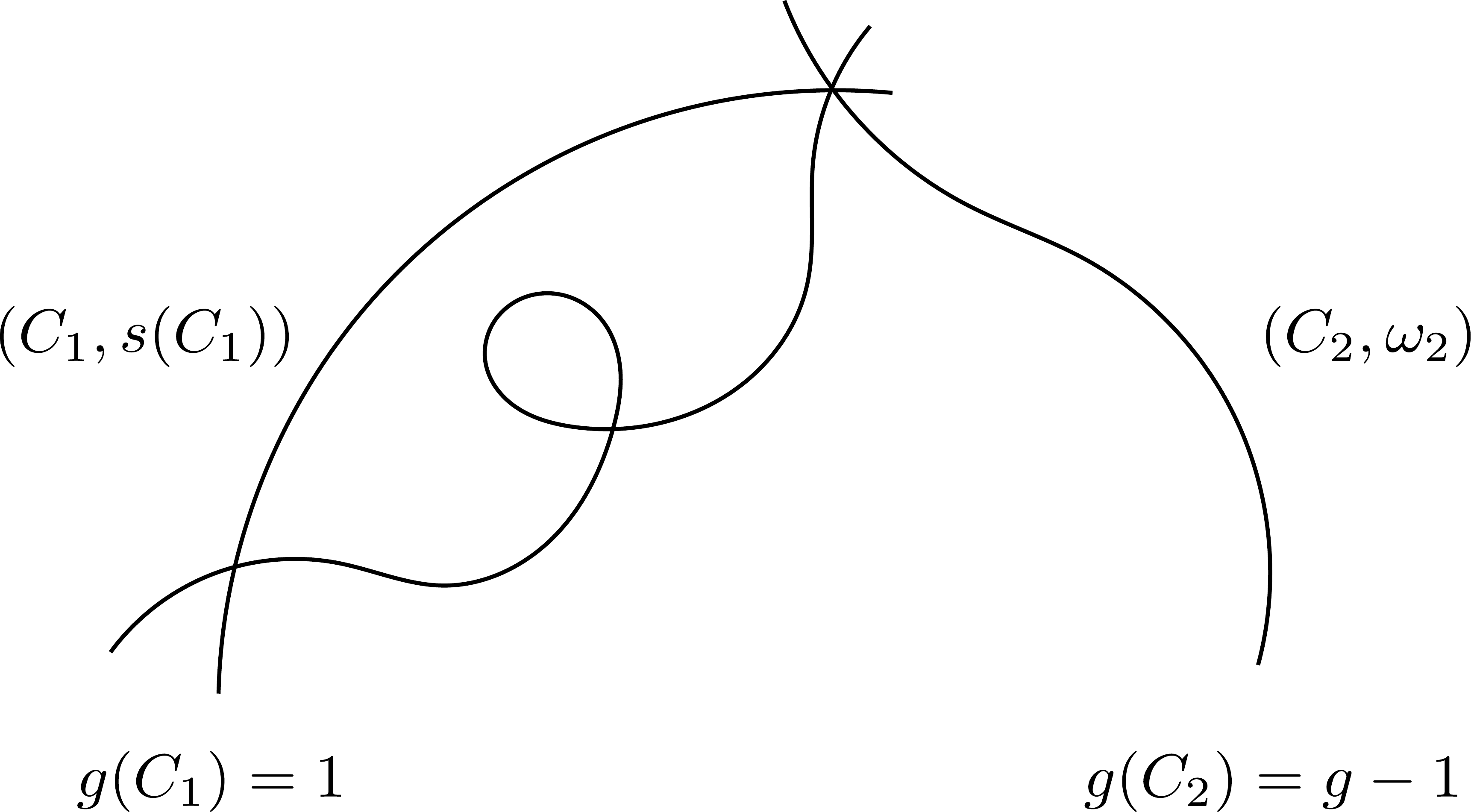
	\caption{Test curve $B$.}
\end{figure}

By the projection formula and standard results from \cite[Chapter 3]{harris1998moduli} we immediately get that $B \cdot \lambda = 1$, $B \cdot \delta_0 = 12$, and $B \cdot \delta_1 = -1$. Let $\omega_{\calX/B}$ be the relative dualizing sheaf of this family $\pi: \calX \ra B$, $\calX_1$ the part of the universal curve corresponding to the pencil of plane cubics, $\calX_2 \cong C_2 \times B$, $\pi_1: \calX_1\ra B$, and $\pi_2: \calX_2 \ra B$. Then, \[\omega_{\calX/B}|_{\calX_1} \cong \pi_1^*\eta \otimes \o_{\calX_1}(Q_1) \otimes \o_{\calX_1}(C_1)\] where $Q_1$ is the section corresponding to the separating node of the family and $C_1$ is the cubic with the zero differential. Intersecting both sides with $Q_1$ and applying $\pi_{1*}$ gives \[Q_1 \cdot (\omega_{\calX_1} \otimes \o_{\calX_1}(Q_1)) = \deg_{\pi} \eta + Q_1^2 + C_1 \cdot Q_1.\] This implies that $\deg_{\pi} \eta = 0$.  It remains to compute $B \cdot D$. \\

\noindent \textbf{Claim}. $(2g-2)!(B \cdot D) = (\varphi_*\pi^*B) \cdot W_{2g-2}$. \\

\noindent \textit{Proof of claim.} Recall from above that $(2g-2)! (B \cdot D) =  \pi^*B \cdot D_{2g-2} = \pi^*B \cdot (\varphi^*W_{2g-2} - E)$. Thus it suffices to show $\pi^*B \cdot E = 0$. Our method is to examine the types of pointed stable differentials appearing in $\pi^{-1}(B)$ and then to enumerate all possible codimension one boundary loci containing such differentials. We then show that these loci are indeed not in $E$. 

Pointed stable differentials in $\pi^{-1}(B)$ have two possible structures, illustrated in Figure 3. First, suppose $B_t$ is a member of the family with nonzero differential on the genus 1 component. Consider the preimage of $B_t$ in $\ov{\calP}(1^{2g-2})$ - that is, the curve with a twice-marked rational component $(R, z_1, z_2)$ inserted at the node, where $z_1, z_2$ are the zeros of a unique differential $\eta_R \sim z_1 + z_2 - 2q^-_1 -2q^-_2$ satisfying $\text{Res}_{q^-_i}\eta_R = 0$ for $i =1,2$ and where $q_1^-$ and $q_2^-$ are the points on $R$ that meet $C_1$ and $C_2$, respectively. (Note that the global residue condition (see \cite{BCGGM}) totally determines $z_1$ and $z_2$ up to automorphism: if we fix $q^-_1$, $q^-_2$, and $z_1$ to be $0$, $\infty$, and $1$ respectively and denote $z_2$ by $a$, then locally at 0 the differential is \[c\frac{(z-1)(z-a)}{z^2}dz.\] Since the residue at 0 is given by $-(a+1)$ and this must be 0, we have that $a = -1$.) Now suppose that the differential on the genus 1 component $C_1$ is zero. Then the preimage of such a curve in $\ov{\calP}(1^{2g-2})$ consists of the curve with markings $z_1$, $z_2$ on $C_1$ such that there exists a differential $\eta_1 \sim z_1 + z_2 - 2q_1$. Note that since $\P H^0(K_{C_1}(2q_1)) \cong \P^1$, this curve parametrized in $B$ has a one-dimensional preimage in $\ov{\calP}(1^{2g-2})$, which we call $E'$. 

\begin{figure}[h!]
	\centering
	\def \svgscale{.27}
	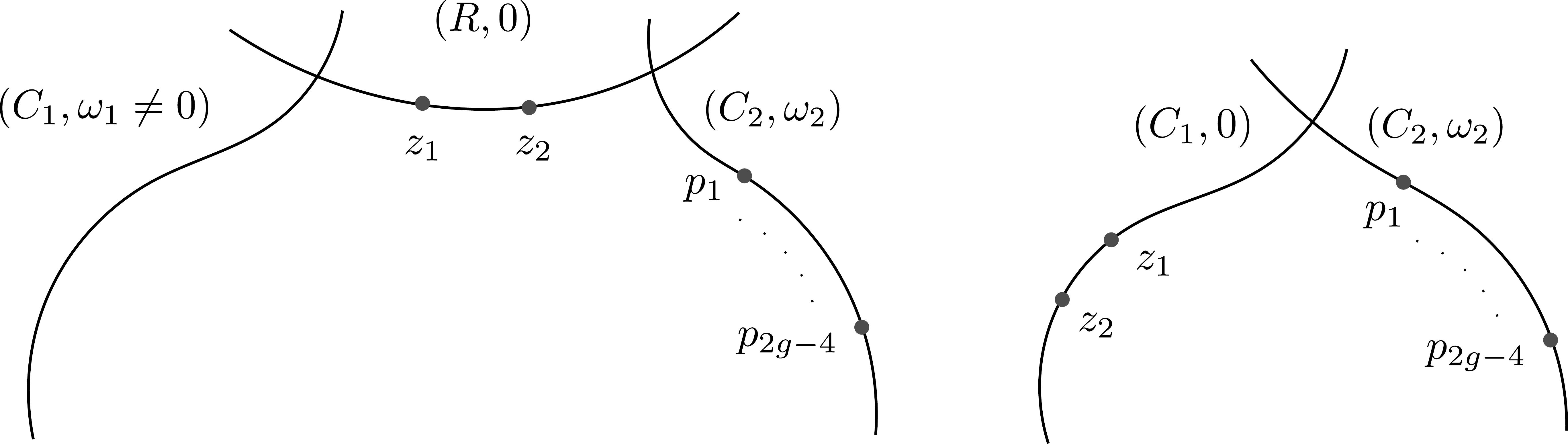
	\caption{Pointed stable differentials in $\pi^{-1}(B)$.}
\end{figure} 

If an irreducible divisor $F$ in $E$ contains a pointed stable differential in $\pi^{-1}(B)$, then generic points of $F$ must parametrize either one or two-nodal curves. Suppose a generic point of $F$ parametrizes one-nodal curves. Using the notation from Lemma 3.2, there are two possibilities for $F$: $E_{g-1}(0)$ and $E_1(0)$. Lemma 3.2 shows that neither of these are in $E$. Now suppose that a generic point in $F$ parametrizes two-nodal curves of the type illustrated in Figure 3. In order to show that such a locus is not contained in $E$, we must exhibit some pointed stable differential of this form which does not get mapped into $W_{2g-2}$. First we pick a $q_2$ which is not a Weierstrass point of $C_2$. Then note that neither $z_1$ nor $z_2$ can be limit Weierstrass points because an admissible cover of degree $g$ for this curve totally ramified at one of these points requires total ramification at $q_2 \in C_2$ (since $q_2$ is not a Weierstrass point of $C_2$), as well as at least simple ramification at $q_1$ in the genus 1 component. Moreover, if we pick a general $\omega_2$, its zeros $p_1, \dots, p_{2g-4}$ also avoid the finitely many limit Weierstrass points on $C_2$, by the same argument used in the proof of Lemma 3.2. This proves the claim.   \\

We now calculate $ \frac{1}{(2g-2)!}(\varphi_*\pi^*B \cdot W_{2g-2})$. When the differential on the genus 1 component is nonzero, none of the marked points can be limit Weierstrass points, by the reasoning in the previous paragraph. Now suppose that the differential on the genus 1 component is zero. There are $g^2-1$ points $p$ satisfying $gp \sim gq_1$ with $p \neq q_1$. Note that if $z_1$ is such a point, then $z_2$ is as well, by the relation $z_1 + z_2 \sim 2q$. Thus, there arise $\frac{g^2-1}{2}$ such differentials and each differential results in a multiplicity 2 intersection with $W_{2g-2}$.  Let us also note that $E'$ occurs with multiplicity 1 in the preimage of $B$. This is a direct consequence of our choice of differential on the genus 1 component by the section $s$, which meets the locus of zero differentials in $\ov{\calH}_1$ with multiplicity 1.  

Putting this all together, $B \cdot D =   \frac{1}{(2g-2)!}(\varphi_*\pi^*B \cdot W_{2g-2}) = g^2-1$. This gives a relation \[ g^2-1 = b + 12c_0 - c_1.\] 

\noindent \textit{Test curve $C$.} \\ 

Consider the following test curve in $\P\ov{\calH}_g$: fix general curves $C_1$, $C_2$ of genus $i$, $1 \leq i \leq \lfloor g/2 \rfloor$, and genus $g-i$ respectively, attached at a node $q_1 \sim q_2$, along with $\omega_1$ and $\omega_2$ general nonzero holomorphic differentials on $C_1$ and $C_2$ respectively, with zeros $z_1, \dots, z_{2i-2}$ and $p_1, \dots, p_{2g-2i-2}$. Now vary the point of attachment $q_2$ in $C_2$ (see Figure 4). 

\begin{figure}[h!]
	\centering
	\def \svgscale{.32}
	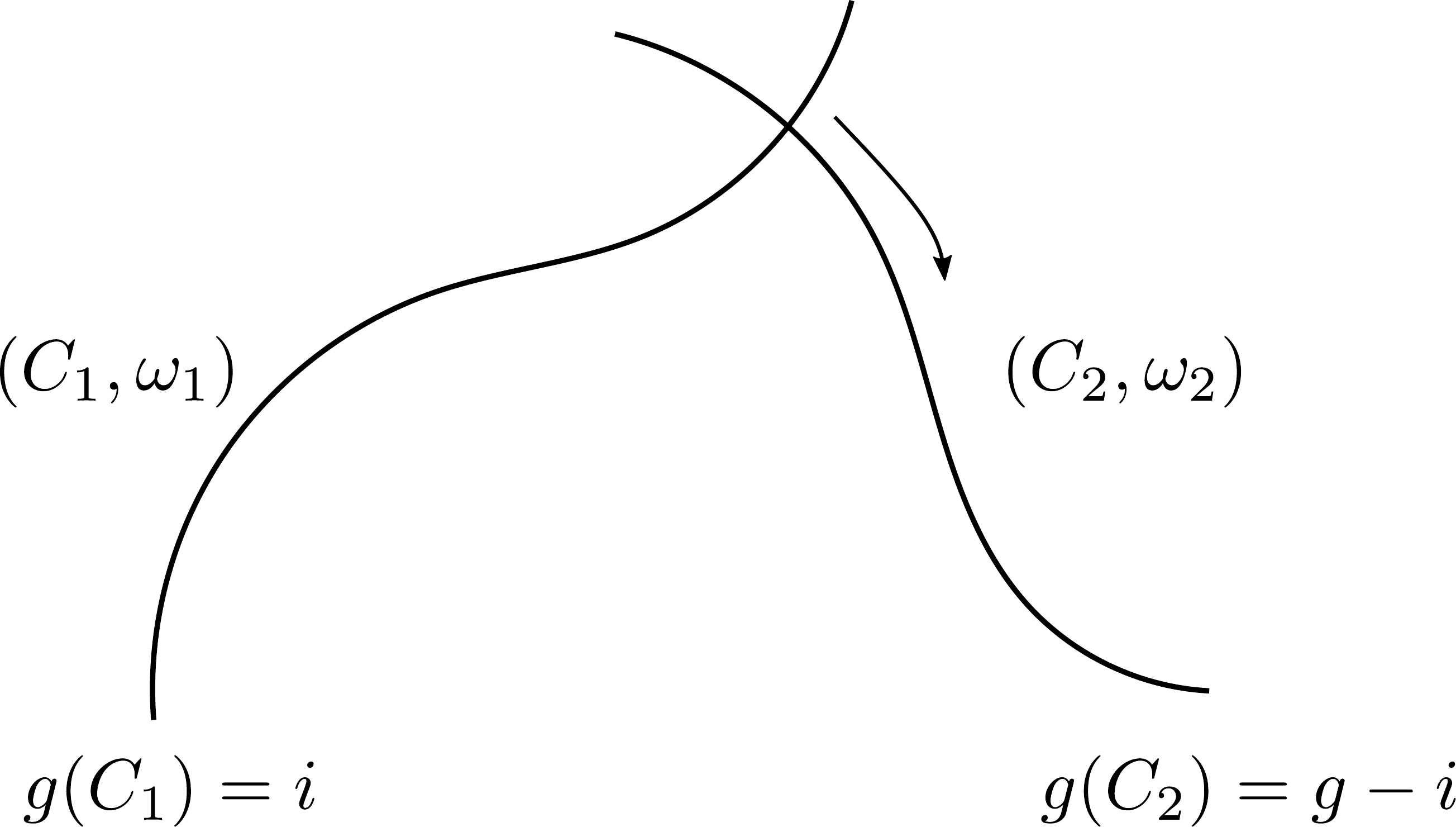
	\caption{Test curve $C$.}
\end{figure}

In order to find $C \cdot \eta$, we intersect the relation \[ \omega_{\calX/C}|_{\calX_1} \cong \pi_1^*\eta \otimes \o_{\calX_1}(Q_1) \otimes \o_{\calX_1}\bigg(\sum_{j=1}^{2i-2} Z_j\bigg)\] with $Q_1$ and apply $\pi_*$. This gives $C \cdot \eta = 0$. Since $C_1$ and $C_2$ are both smooth $C \cdot \delta_0 = 0$. Moreover, for $C_q$ a curve in this family, $H^0(K_{C_q}) = H^0(K_{C_1}) \oplus H^0(K_{C_2})$ and since this is independent of $q$, we have $C \cdot \lambda = 0$. This entire family is contained in $\delta_i$, so \[C \cdot \delta_i = \deg(N_{{q \times C_2}/{C_1 \times C_2}} \otimes N_{\Delta/C_2 \times C_2})  = \Delta^2 = 2-2(g-i).\] 
Hence, \[C \cdot D = (2-2g+2i)c_i.\] \\

\noindent \textbf{Claim}. $(2g-2)!(C \cdot D) = \varphi_*\pi^*C \cdot W_{2g-2} = (2i-2)B_1 \cdot W + (2g-2i-2)B_2 \cdot W + 2B_3 \cdot W$ where $B_1$, $B_2$, $B_3$ are the curve classes in $\ov{\M}_{g, 1}$ illustrated in Figure 5. \\

\begin{figure}[h!]
	\centering
	\def \svgscale{.28}
	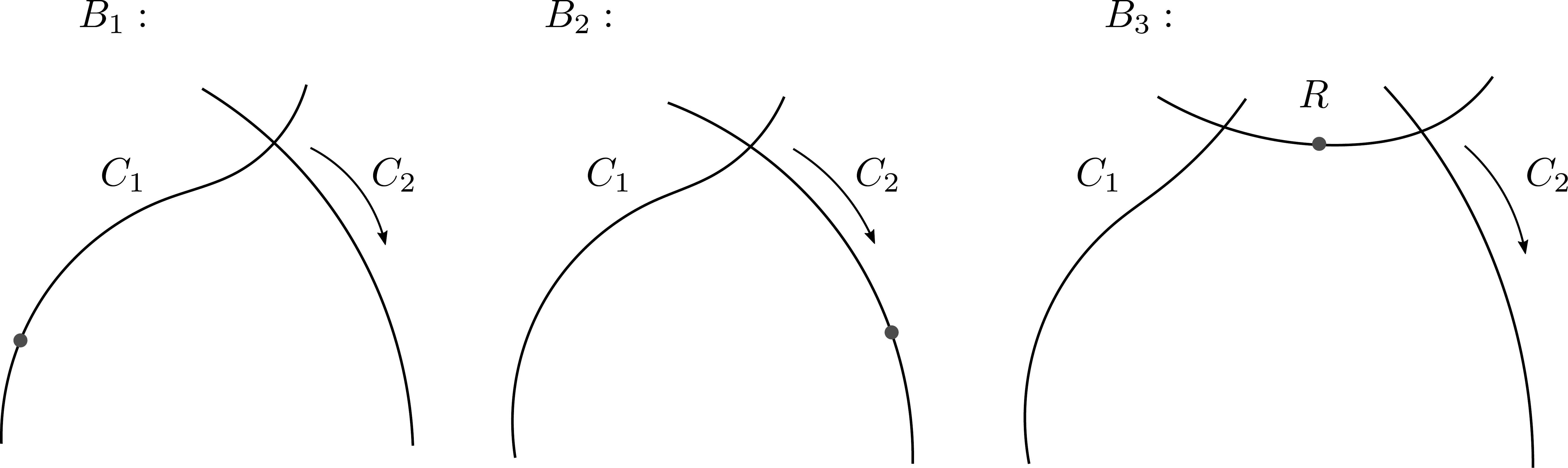
	\caption{Curves $B_1$, $B_2$, $B_3$ in $\ov{\M}_{g,1}$.}
\end{figure}

\noindent \textit{Proof of claim.} Again the goal is to show that $\pi^*C \cdot E = 0$ and we use the same method as above. Consider the morphism $\pi: \ov{\calP}(1^{2g-2}) \ra \P\ov{\calH}_g$. Let $(C_q, \omega= (\omega_1, \omega_2))$ be a stable differential parameterized in $C$ where $q$ coincides with none of the zeros of $\omega_2$. The preimage of such a stable differential under $\pi$ is the collection of data $(C_q', \omega' = (\omega_1, 0, \omega_2), z_1, \dots , z_{2i-2}, s_1, s_2, p_1, \dots, p_{2g-2i-2})$ where $C_q'$ is the curve $C_q$ along with a twice-marked rational component $(R, s_1, s_2)$ inserted at the node, and where $s_1$ and $s_2$ are the zeros of some differential $\eta_R \sim s_1 + s_2 - 2q^-_1 - 2q^-_2$ satisfying $\text{Res}_{q^-_1}\eta_R = \text{Res}_{q^-_2}\eta_R = 0$, as before. When $q$ meets some $p_i$, we blow up to get a curve with an additional rational component $S$, which bears a twisted differential of type $(1, -3)$ with its pole at the point where $S$ meets $C_2$ (see Figure 6). 

\begin{figure}[h!]
	\centering
	\def \svgscale{.28}
	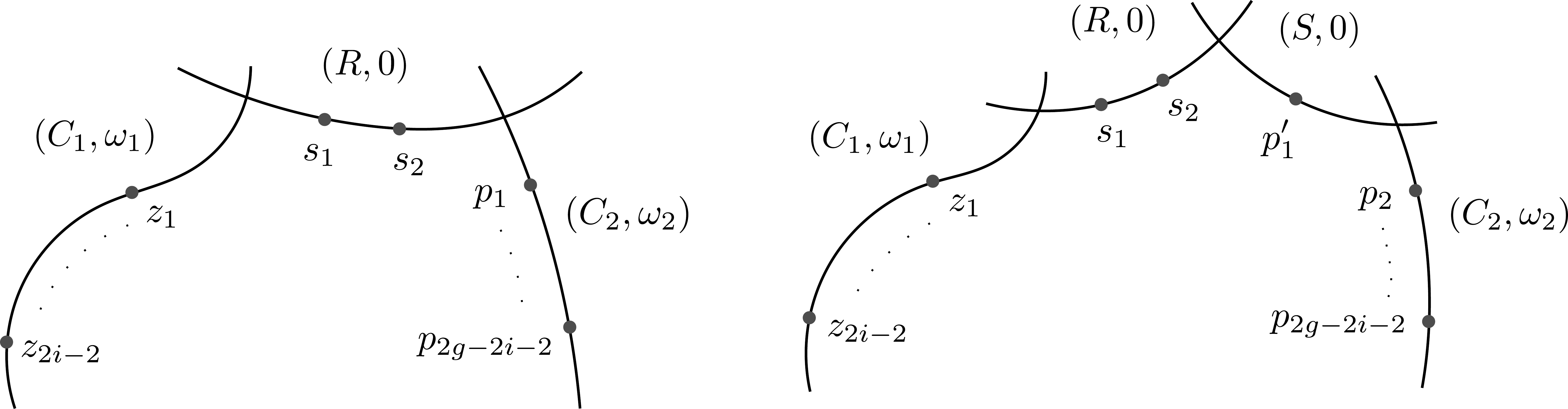
	\caption{Pointed stable differentials in $\pi^{-1}(C)$.}
\end{figure}

If an irreducible divisor $F$ in $E$ contains a pointed stable differential in $\pi^{-1}(C)$, then generic points of $F$ must parametrize either one, two, or three-nodal curves. Suppose a generic point of $F$ parametrizes one-nodal curves. The possibilities are the loci $E_i(0)$, $E_{g-i}(0)$, and $E_{g-i}(1)$. By Lemma 3.2, none of these are in $E$. Now suppose that a generic point of $F$ parametrizes two-nodal curves of the type illustrated on the left in Figure 6. Using the same strategy as with test curve $B$, we will exhibit a pointed stable differential in this locus which does not get mapped into $W_{2g-2}$. We first pick $q_1$ and $q_2$ not Weierstrass points of $C_1$ and $C_2$ respectively, as well as general $\omega_1$ and $\omega_2$ whose zeros avoid the finitely many limit Weierstrass points on each component. Likewise, $s_1$ and $s_2$ cannot be limit Weierstrass points for the same reason as before. This shows that such a locus is not in $E$. Finally, any other possible two or three-nodal locus containing pointed stable differentials in $\pi^{-1}(C)$ is of codimension 2 in $\ov{\calP}(1^{2g-2})$, since we require the differential on $C_2$ to have a zero at the point of attachment to the rational bridge. 

The divisors $W_{2g-2}$ and $\sum_{i=1}^{2g-2}f_i^*W$ can only possibly differ on loci in $\ov{\M}_{g, 2g-2}$ where the curves have rational components which are stable due to the presence of 2 or more marked points. Indeed, let $X$ be a nodal curve with a genus $g$ component $C$ attached to a genus 0 component $R$ at a point $p$ and let $z_1$ be one of the marked points on $R$. If $z_1$ is a limit Weierstrass point of $C$, then $p$ must be a Weierstrass point of $C$, and so this curve maps into $W$. So we conclude that $W_{2g-2}$ and $\sum_{i=1}^{2g-2}f_i^*W$ do not differ. Thus, \begin{align*} \varphi_*\pi^*C \cdot W_{2g-2} & = \varphi_*\pi^*C \cdot \sum_{i=1}^{2g-2}f_i^*W \\ & = \sum_{i=1}^{2g-2}f_{i*}\varphi_*\pi^*C \cdot W \\ & = (2i-2)B_1 \cdot W + (2g-2i-2)B_2 \cdot W + 2B_3 \cdot W.\end{align*} This proves the claim. \\

Recall that \[ W = \frac{g(g+1)}{2}\psi - \lambda - \sum_{i=1}^{g-1}\frac{(g-i)(g-i+1)}{2}\delta_i\] where $\delta_i$ now denotes the boundary divisor with the genus $i$ component marked \cite{cukierman1989}.
Let us first find $B_2 \cdot W$. When $q$, the node, meets the marked point $p$, we blow up and get that $B_2 \cdot \delta_i = 1$. Let $P$ be the section corresponding to the marked point before the blow up procedure, $\pi$ be the associated blow up map, and $\wt{P}$ the strict transform of $P$. Then \[B_2 \cdot \psi = -\wt{P}^2 = -(\pi^*P - E)^2 = -P^2 + 1 = 1.\] Since $B_2 \subset \delta_{g-i}$, \[B_2 \cdot \delta_{g-i} = \deg(N_{{q \times C_2}/C_1 \times C_2} \otimes N_{\wt{\Delta}/S}) = 1-2g+2i ,\] where $S$ is $C_2 \times C_2$ blown up at the point $(p, p)$. Finally, we want to compute $B_2 \cdot \lambda$. Let $\phi: \ov{\M}_{g,1} \ra \ov{\M}_g$. Then, $B_2 \cdot \phi^*\lambda = \phi_*B_2 \cdot \lambda = 0$, since the Hodge bundle does not vary when $q$ varies in $C_2$, as explained above. This gives $B_2 \cdot W = i(g-i)(i+2)$. 

Now we compute $B_1 \cdot W$. By similar reasoning $B_1 \cdot \psi = B_1 \cdot \lambda = 0$. Moreover, \[B_1 \cdot \delta_i = \Delta^2 = 2-2(g-i).\] Thus, \[ B_1 \cdot W = (g-i-1)(g-i)(g-i+1).\] Similarly, $B_3 \cdot W = (g-i-1)(g-i)(g-i+1)$. Putting this all together, we get that \[ c_i = (g+3)i(i-g).\] Using test curve $B$, we also find \[ c_0 = -\frac{g(g+1)}{2}.\] This completes the proof of Theorem 1.1.
\end{proof}
\vskip .5cm

\noindent We will now check the divisor class formula for $D$ for a couple curve classes in low genus. 

\begin{example} We will verify the above divisor class for a curve class $B$ consisting of a general pencil of plane quartics with canonical divisors given by a fixed general line $L \subset \P^2$. In genus 3 \[ D = -24\eta + 68\lambda - 6\delta_0 - 12\delta_1.\] By standard calculations in \cite[Chapter 3]{harris1998moduli}, we have that $B \cdot \lambda = 3$, $B \cdot \delta_0 = 27$, and $B \cdot \delta_1 = 0$. Let $\pi: \calX \ra B$ be the total space of the family. Intersecting the relation \[ \omega_{\calX/B} = \pi^*\o_B(-1) \otimes \o_{\calX}(L)\] with $E$, the blow up of a basepoint of the pencil, and then applying $\pi_*$ gives \[-E^2 = \text{deg}_{\pi}\eta + (L \cdot E). \] Hence, $\text{deg}_{\pi}\eta = 1$. So, $B \cdot D = 18$. 

On the other hand, we have that the degree of the curve in $\P^2$ traced out by the flex points of a general pencil of degree $d$ curves is $6d-6$ (see \cite[Section 11.3]{eisenbud20163264}). Thus, we have verified that indeed $B \cdot D = 18$.

\end{example}
\vskip .5cm

\begin{example} We will also verify the divisor class for a curve class $B$ consisting of a general pencil of genus 4 canonical curves in a quadric $Q \subset \P^3$ with canonical divisors determined by a fixed hyperplane $H$.  In genus 4, \[ D = -60\eta + 114\lambda - 10\delta_0-21\delta_1-28\delta_2.\] By the same reasoning as above, $B \cdot \eta = 1$. We can consider $\calX \subset \P^1 \times Q$ as a divisor of type $\o_{\P^1 \times Q}(1, 3)$ and $\P^1 \times Q \subset \P^1 \times \P^3$ as a divisor of type $\o_{\P^1 \times \P^3}(0, 2)$. Then, using adjunction, we find that \[ \omega_{\P^1 \times Q} = \o_{\P^1 \times Q}(-2, -2).\] Using adjunction again, \[ \omega_{\calX} = \o_{\calX}(-1, 1).\] Thus, \[ \omega_{\calX/B} = \o_{\calX}(1, 1).\] Let $\pi_1$ and $\pi_2$ be the two projections associated to the product $\P^1 \times \P^3$. Let $\alpha$ be the divisor class $\pi_1^*\o(1)$ and let $\beta$ be the divisor class $\pi_2^*\o(1)$. Then 
\[B \cdot \kappa = (\alpha + \beta)(\alpha + \beta)(\alpha + 3\beta)(2\beta) = 14.\] Moreover, $B \cdot \delta_0 = 34$ (see \cite[Section 7.4]{eisenbud20163264}), and $B \cdot \delta_i = 0$ for $i >0$. Using the relation \[ \lambda = \frac{\kappa + \delta}{12}\] we find that $B \cdot \lambda = 4$. Thus, \[ B \cdot D = 56.\] 

We can also see that $B \cdot D = 56$ by a standard Porteous formula calculation as in \cite{cukierman1989}. Let $\pi: \calX \ra B$ be the total space of the family. Note that since the curve spanned by the Weierstrass points is codimension 1, we need not worry about having singular fibers. Let $\calX_2 = \calX \times_{B} \calX$ and let $\pi_1$ and $\pi_2$ be the two projections to $\calX$. Let 
\begin{align*} 
& E = (\pi_1)_*(\pi_2^*\omega_{\calX/B}) \\
& F = (\pi_1)_*(\pi_2^*\omega_{\calX/B} \otimes \o_{\calX_2}/\mathcal{I}_{\Delta}^4).
\end{align*}
Now regard $E$ and $F$ as vector bundles over just the smooth locus of $\calX$ and consider the morphism $\varphi: E \ra F$. We are interested in the locus $W_B$ on $\calX$ where $\text{rank}(\varphi_x) \leq 3$. Note that $F$ is the bundle of relative principal parts $P^4(L)$ where $L = \omega_{\calX/B}$. By the standard exact sequence \[0 \ra L \otimes \text{Sym}^{k-1}(\omega_{\calX/B}) \ra P^k(L) \ra P^{k-1}(L) \ra 0, \] we have \[ c_1(F) =  c_1(P^4(L)) = \sum_{k=1}^{4} (c_1(L) +(k-1)(c_1(\omega_{\calX/B}))) = 10\omega_{\calX/B}.\] Let $f$ be the class of a fiber of $\pi$. Applying Porteous' formula gives us that 
\begin{align*} W_B & = c_1(F) - c_1(E) \\
& = 10\omega_{\calX/B} - \pi^*\lambda \\
& = 10\omega_{\calX/B} - 4f.
\end{align*}
Let $\text{bl}: \calX \ra Q$ be the blow up morphism and $E_1, \dots, E_{18}$ the exceptional divisors over the basepoints of the pencil. Intersecting $W_B$ with $\text{bl}^*(l_1 + l_2) = \frac{1}{3}f + \frac{1}{3}\sum_{i=1}^{18}E_i$ shows that the degree of the curve traced out by the Weierstrass points of the pencil is 56. 

\end{example}
\vskip .5cm

\section{The extremality of $\P\ov{\calH}_g(2)$ and $\P\ov{\mathcal{Q}}_g(2)$}

In this section we will prove Theorem 1.2. We will use the following condition from \cite{chen2013strata} to check the extremality of $\P\ov{\calH}_g(2) \subset \P\ov{\calH}_g$ and $\P\ov{\mathcal{Q}}_g(2) \subset \P\ov{\mathcal{Q}}_g$.
\vskip .5cm
\begin{lemma}{\cite[Proposition 4.1]{chen2013strata}} Suppose $D$ is an irreducible effective divisor and $A$ is a big divisor in a projective variety $X$. Let $S$ be a set of irreducible effective curves contained in $D$ such that the union of these curves is Zariski dense in $D$. If for every curve $C$ in $S$ we have \[C \cdot (D + dA) \leq 0 \] for a fixed $d > 0$, then $D$ is an extremal divisor in the pseudoeffective cone $\ov{\text{Eff}}^1(X)$, i.e, if for any linear combination $D = D_1 + D_2$ with $D_i$ pseudoeffective, $D$ and $D_i$ are proportional. 
\end{lemma}

\begin{proof}[Proof of Theorem 1.2] Recall that \[ \P\ov{\calH}_g(2) = 24\lambda - (6g-6)\eta - 2\delta_0 - 3\sum_{i=1}^{\lfloor g/2 \rfloor} \delta_i.\] Let $C$ be the closure of a Teichm\"uller curve generated by some $(X, \omega) \in \P\ov{\calH}_g(2)$. Let $\chi = 2g(C) - 2 + |\Delta| = 2C \cdot \eta$ \cite{moller2006variations}, where $|\Delta|$ is the number of cusps in $C$, and $L$ the sum of its first $g$ Lyapunov exponents. We are concerned with the partition $\mu = (m_1, \dots, m_n) = (2, 1^{2g-4})$. Using  \cite[Proposition 4.8]{nonvaryinglyapsum} we have that 
\begin{align*}
C \cdot \lambda &= \frac{\chi}{2}L \\
C \cdot \delta_0 &= \frac{\chi}{2}(12L - 12\kappa_{\mu})
\end{align*}
where $\kappa_{\mu} = \frac{1}{12}\sum_{j=1}^n \frac{m_j(m_j+2)}{m_j+1}$. Since Teichm\"uller curves do not intersect higher boundary divisors (see \cite[Corollary 3.2]{nonvaryinglyapsum}), \[ C \cdot \delta_i = 0 \; \; \text{for} \; \; i>0.\] 

\noindent So, \[ C \cdot \P\ov{\calH}_g(2) = -\frac{\chi}{3}.\] Moreover, by \cite[Proposition 4.8]{nonvaryinglyapsum} \[ C \cdot \psi_i = \frac{C \cdot \lambda - (C \cdot \delta)/12}{(m_i + 1)\kappa_{\mu}} = \frac{\chi}{2(m_i+1)}.\] Since $\psi_i$ has positive degree on nonconstant families \cite[Chapter 6]{harris1998moduli} $\chi > 0$. Now let $A$ be an ample divisor in $\P\ov{\calH}$. We write \[ A = a\lambda + b\eta + \sum_{i=0}^{\lfloor g/2 \rfloor}c_i\delta_i.\] We must now choose a sufficiently small value $d$ such that $C \cdot (\P\ov{\calH}_g(2) + dA) \leq 0$ for all Teichm\"uller curves in $\P\ov{\calH}_g(2)$. Let \[ d = \inf_{\substack{ \text{Teichm\"uller curves} \\  \text{in}  \; \P\ov{\calH}_g(2)}} \Bigg\{ \frac{2}{3(b-12c_0\kappa_{\mu} + (a+12c_0)L)} \Bigg\} \] The expression in the brackets comes from solving for $d$ in $C \cdot (\P\ov{\calH}_g(2) + dA) = 0$ using the intersection information given above. Since $C \cdot A > 0$ ($A$ is ample) and $C \cdot \P\ov{\calH}_g(2) < 0$ for all Teichm\"uller curves $C$, the expression in the brackets will always be positive and will only depend on $L$. Moreover, the infimum may never be zero since the sum of Lyapunov exponents $L$ has a uniform upper bound $g$. Since Teichm\"uller curves in any stratum are Zariski dense, we have shown that $\P\ov{\calH}_g(2)$ is extremal by Lemma 4.1. 

Recall from \cite{korotkin2013tau} that \[ \P\ov{\mathcal{Q}}_g(2) = 72\lambda-10(g-1)\eta - 6\sum_{i=0}^{\lfloor g/2 \rfloor}\delta_i. \] Here we denote the partition $(d_1, \dots, d_n) = (2, 1^{4g-6})$. Let $C$ be a Teichm\"uller curve generated by a half translation surface $(X, q) \in \P\ov{\mathcal{Q}}_g(2)$. Let $L^+$ be the sum of the involution invariant Lyapunov exponents (see \cite{eskin2014sum} and \cite[Section 2.2]{quadlyapsum} for background material) and let $\chi = 2g(C) - 2 + |\Delta| = C \cdot \eta$ \cite{moller2006variations}. From \cite{eskin2014sum} we can write \[ L^+ = c_{\text{area}} + \kappa, \; \; \text{where} \; \; \kappa = \frac{1}{24} \sum_{j=1}^n \frac{d_j(d_j + 4)}{d_j + 2}\] and $c_{\text{area}}$ is the area Siegel-Veech constant of $(X, q)$. By \cite[Proposition 4.2]{quadlyapsum} \begin{align*} & C \cdot \lambda = \frac{\chi}{2}(c_{\text{area}} + \kappa) \\ & C \cdot \delta = 6\chi\cdot c_{\text{area}}. \end{align*}  

\noindent Hence, \[ C \cdot \P\ov{\mathcal{Q}}_g(2) = -\frac{\chi}{2}.\] If $a\lambda + b\eta + \sum_{i=0}^{\lfloor g/2 \rfloor}c_i\delta_i$ is an ample divisor, we can ensure that all coefficients for the boundary divisors are the same by adding on an appropriate effective divisor of boundary divisors. This gives us a big divisor \[A = a\lambda + b\eta + c\delta \] where $c = \max_i \{ c_i\}$. Note that when $C \cdot A \leq 0$, any $d >0$ satisfies the condition in Lemma 4.1. So assuming $C \cdot A > 0$, we set \[ d = \inf_{\substack{ \text{Teichm\"uller curves} \\  \text{in}  \; \P\ov{\mathcal{Q}}_g(2) \; \text{with} \; C \cdot A >0}} \Bigg\{ \frac{1}{2b + 12c_{\text{area}}c + aL^+} \Bigg\}.\] The expression in the brackets comes from solving for $d$ in the expression $C \cdot (\P\ov{\mathcal{Q}}_g(2) + dA) = 0$. Since $c_{\text{area}}$ is bounded from above and $L^+ = c_{\text{area}} + \kappa$, $d$ is positive and so $\P\ov{\mathcal{Q}}_g(2)$ is extremal by Lemma 4.1. 

\end{proof}

\noindent For completeness we include the following proposition.

\begin{prop} The boundary divisors $\delta_i$, $0 \leq i \leq \lfloor g/2 \rfloor$, span extremal rays in $\ov{\text{Eff}}^1(\P\ov{\calH}_g)$ and in $\ov{\text{Eff}}^1(\P\ov{\mathcal{Q}}_g)$.
\end{prop}

\noindent We will use the following well-known condition to check the extremality of the boundary divisors. 

\begin{lemma}{\cite[Lemma 4.1]{chen2014genusone}} Suppose that $C$ is a moving curve in an irreducible effective divisor $D$ of a projective variety $X$. Suppose that $C$ satisfies $C \cdot D < 0$. Then $D$ is extremal. 
\end{lemma}

\begin{proof}[Proof of Proposition 4.2] 
Let $f: \P\ov{\calH}_g \ra \ov{\M}_g$. We will use the following strategy. We will find moving curves in each of the irreducible boundary divisors of $\ov{\M}_g$ which satisfy the condition of Lemma 4.3. Then let $C$ be such a moving curve and let $D$ be an irreducible boundary divisor in $\ov{\M}_g$. Given a point in $f^{-1}(C)$, we can find a curve $C'$ through it by taking the intersection of $g-1$ hyperplane classes $H_1, \dots, H_{g-1}$ in $f^{-1}(C)$. When the choice of these hyperplane classes is general, $C'$ is irreducible by Bertini's theorem and moreover $C'$ covers $C$. To see the latter statement, note that as a result of the irreducibility of $C'$ we just need to show that the image of $C'$ under $f$ is not a single point $p \in C$. Note that $f^{-1}(p)$ is a divisor in $f^{-1}(C)$ and so it must intersect $\cap_{i=1}^{g-1} H_i$ positively. Thus $C'$ covers $C$. Finally, since varying the hyperplane classes used to construct $C'$ will not change the numerical equivalence class of $C'$, we know that it is indeed a moving curve. Thus, $f^*D \cdot C' = D \cdot f_*C' < 0$ and we can conclude by Lemma 4.3 that $f^*D$ is extremal. All that remains is to find appropriate moving curves in each of the boundary divisors of $\ov{\M}_g$.

Let $X$ be the following curve in $\delta_0 \subset \ov{\M}_g$: take a genus $g-1$ curve $C$ and identify a fixed point $p$ of $C$ to a varying point $q$ of $C$. This is a moving curve in $\delta_0$ and \[X \cdot \delta = \deg(N_{\wt{\Delta}/S} \otimes N_{\wt{C \times p}/S}) = \Delta^2 - 1 = 3-2g\] where $S$ is the blow up of $C \times C$ at $(p, p)$ and $\wt{\Delta}$ and $\wt{C \times p}$ denote the proper transforms. Since $X \cdot \delta_1 = 1$, we also have $X \cdot \delta_0 = 2-2g < 0$.  

Now assume that $g \geq 3$ and let $X$ be the moving curve in $\delta_i \subset \ov{\M}_g$ given by attaching a general genus $i$ curve $C_1$ to a general genus $g-i$ curve $C_2$ and varying the point of attachment in  $C_2$. In the computation for test curve $C$ in the proof of Theorem 1.1 we explained that $X \cdot \delta_i = 2-2(g-i) < 0$. When $g=2$, we can choose our moving curve $X$ in $\delta_1$ to be the family given by attaching a pencil of plane cubics to a general genus 1 curve. In the computation for test curve $B$ in the proof of Theorem 1.1, we explained that $X \cdot \delta_1 = -1$. Thus, we have found all necessary moving curves.  

By the same argument we can also conclude that the boundary divisors span extremal rays in $\ov{\text{Eff}}^1(\P\ov{\mathcal{Q}}_g)$ as well.

\end{proof}

\bibliographystyle{alpha}
\bibliography{divisors}

\end{document}